\documentclass[12pt,a4paper]{amsart}
\usepackage{amsmath,amsthm,amssymb,amsfonts}
 
\usepackage[dvipsnames]{xcolor}
\usepackage[bookmarks=true,hyperindex,pdftex,colorlinks,citecolor=blue,linkcolor=blue,urlcolor=blue]
{hyperref}
\usepackage[export]{adjustbox}
\usepackage{graphicx}
\usepackage{subcaption}
\usepackage{mathtools}

\theoremstyle{definition}
\newtheorem{theorem}{Theorem} 
\newtheorem{cor}[theorem]{Corollary}
\newtheorem{prop}[theorem]{Proposition}
\newtheorem{lemma}[theorem]{Lemma}

\newtheorem{mydef}[theorem]{Definition}

\newtheorem{question}[theorem]{Question}
\newtheorem{rem}[theorem]{Remark}
\newtheorem{remark}[theorem]{Remark}

\newcommand{\ds}{\displaystyle}
\newcommand{\R}{\mathbb{R}}
\newcommand{\N}{\mathbb{N}}

\newcommand{\Z}{\mathcal{Z}}
\newcommand{\C}{\mathbb{C}}
\newcommand{\K}{\mathbb{K}}

\newcommand{\e}{\varepsilon}
\newcommand{\eps}{\varepsilon}

\DeclareMathOperator{\re}{Re}

\DeclareMathOperator{\NA}{NA}
\DeclareMathOperator{\NRA}{NRA}
\DeclareMathOperator{\Id}{Id}

\renewcommand{\leq}{\leqslant}
\renewcommand{\geq}{\geqslant}

\renewcommand{\geq}{\geqslant}

\begin{document}
\title{On various types of density of numerical radius attaining operators}

\author[Dantas]{Sheldon Dantas}
\address[Dantas]{Department of Mathematics, Faculty of Electrical Engineering, Czech Technical University in Prague, Technick\'a 2, 166 27, Prague 6, Czech Republic \newline
	\href{http://orcid.org/0000-0001-8117-3760}{ORCID: \texttt{0000-0001-8117-3760} } }
\email{\texttt{gildashe@fel.cvut.cz}}

\author[Kim]{Sun Kwang Kim}
\address[Kim]{Department of Mathematics, Chungbuk National University, 1 Chungdae-ro, Seowon-Gu, Cheongju, Chungbuk 28644, Republic of Korea \newline
	\href{http://orcid.org/0000-0002-9402-2002}{ORCID: \texttt{0000-0002-9402-2002}  }}
\email{\texttt{skk@chungbuk.ac.kr}}

\author[Lee]{Han Ju Lee}
\address[Lee]{Department of Mathematics Education, Dongguk University - Seoul, 04620 (Seoul), Republic of Korea \newline
	\href{http://orcid.org/0000-0001-9523-2987}{ORCID: \texttt{0000-0001-9523-2987}  }
}
\email{\texttt{hanjulee@dongguk.edu}}

\author[Mazzitelli]{Martin Mazzitelli}
 \address[Mazzitelli]{Instituto Balseiro, CNEA-Universidad Nacional de Cuyo, CONICET, Argentina.}
 \email{\texttt{martin.mazzitelli@ib.edu.ar}}

\begin{abstract} In this paper, we are interested in studying two properties related to the denseness of the operators which attain their numerical radius: the Bishop-Phelps-Bollob\'as point and operator properties for numerical radius (BPBpp-nu and BPBop-nu, respectively). We prove that every Banach space with micro-transitive norm and second numerical index strictly positive satisfy the BPBpp-nu and that, if the numerical index of $X$ is 1, only one-dimensional spaces enjoy it. On the other hand we show that the BPBop-nu is a very restrictive property: under some general assumptions, it holds only for one-dimensional spaces. We also consider two weaker properties, the local versions of BPBpp-nu and BPBop-nu, where the $\eta$ which appears in their definition does not depend just on $\eps > 0$ but also on a state $(x, x^*)$ or on a numerical radius one operator $T$. We address the relation between the local BPBpp-nu and the strong subdifferentiability of the norm of the space $X$. We show that finite dimensional spaces and $c_0$ are examples of Banach spaces satisfying the local BPBpp-nu, and we exhibit an example of a Banach space with strongly subdifferentiable norm failing it. We finish the paper by showing that finite dimensional spaces satisfy the local BPBop-nu and that, if $X$ has strictly positive numerical index  and has the approximation property, this property is equivalent to finite dimensionality.

\end{abstract}

\thanks{The first author was supported by the project OPVVV CAAS CZ.02.1.01/0.0/0.0/16\_019/0000778 and by the Estonian Research Council grant PRG877. The second author was supported by the National Research Foundation of Korea(NRF) grant funded by the Korea government(MSIT) [NRF-2020R1C1C1A01012267]. The third author was supported by Basic Science Research Program through the National Research Foundation of Korea(NRF) funded by the Ministry of Education, Science and Technology [NRF-2020R1A2C1A01010377]. The fourth author was partially supported by CONICET PIP 11220130100329CO}


\maketitle

\section{Introduction}

E. Bishop and R. Phelps asked if it was possible to extend their result on denseness of norm attaining functionals to bounded linear operators (see \cite{BP}). J. Lindenstrauss, in \cite{Lind}, was the one who gave a negative answer for this question opening possibilities to develop a whole new theory with very elegant and deep results in connection to the geometry of the involved Banach spaces. Parallel to this, I.D. Berg and B. Sims initiated in \cite{BS} the study of the numerical radius attaining operators. Recall that an operator $T$ on $X$ attains the numerical radius if there are $x_0\in S_X$ and $x_0^*\in S_{X^*}$ such that $x_0^*(x_0) = 1$ and $|x_0^*(Tx_0)| = \sup |x^*(Tx)|$, the supremum being taken over all $x\in S_X$ and $x^*\in S_{X^*}$ such that $x^*(x) = 1$. Although for the numerical radius one has to deal with the extra condition $x^*(x) = 1$, and the techniques seem to require much more ingenuity and ability, it attracted the attention of many authors to see when the set of all numerical radius attaining operators is norm dense in the set of bounded operators. For instance, C.S. Cardassi proved that such denseness holds for classical Banach spaces as $\ell_1$, $c_0$, and $L_1(\mu)$ as well as for uniformly smooth Banach spaces (see \cite{C1,C2,C3}). We emphasize the fact that the denseness holds on every Banach space with the Radon-Nikod\'ym property (see \cite[Theorem 2.4]{AP}) but it does not hold in general (see \cite[Section 2]{P}).

After the Bishop-Phelps theorem had been shown, B. Bollob\'as in \cite{B} improved the theorem in the following sense: for given norm one elements $x$ and $x^*$ such that $x^*(x) \approx 1$, it is possible to get new elements $y$ and $y^*$ such that $y^*$ attains the norm at $y$, $y \approx x$, and $y^* \approx x^*$. That is, one can control the distances between the points and the functionals simultaneously. Since Bollob\'as' theorem is no longer true for operators due to Lindenstrauss' results, M. Acosta, R. Aron, D. Garc\'ia, and M. Maestre introduced a new property, which opened even more possibilities to develop the theory, called the {\it Bishop-Phelps-Bollob\'as property} (see \cite{AAGM}). A pair  of Banach spaces $(X, Y)$ satisfies the Bishop-Phelps-Bollob\'as property if it is possible to get a Bollob\'as' theorem for bounded linear operators from $X$ into $Y$, that is, for given norm one bounded linear operator $T$ and element $x$ such that $\|Tx\| \approx 1$, then there are new norm one bounded linear operator $S$ and $x_0$ such that $S$ attains the norm at $x_0$, with $x_0 \approx x$ and $S \approx T$. It is clear that this property implies the denseness of all operators that attain the norm. Again, parallel to the study of the Bishop-Phelps-Bollob\'as property, the study of this property for numerical radius was initiated (see \cite{F, GK, KLM}) and it brings us to the main topic of this paper. In order to detail it more precisely, we introduce the necessary notation and background.

Let $X$ be a Banach space over the scalar field $\K$, which can be either the real numbers $\R$ or the complex numbers $\C$. We denote by $S_X$ the unit sphere of $X$. The Banach space of all bounded linear operators from $X$ into itself is denoted by $\mathcal{L}(X)$ with the operator norm $\|T\|:= \sup \{\|T(x)\|: x \in S_X\}$ for each $T\in \mathcal{L}(X)$. Especially, the dual of $X$ is written as $X^*$. We define the set of states of $X$ by
\begin{equation*} \Pi(X)=\{(x,x^*)\in S_X\times S_{X^*}:~x^*(x)=1\}.
\end{equation*} 
The \emph{numerical radius} of $T \in \mathcal{L}(X)$ and the \emph{numerical index} of $X$ are defined, respectively, by
\begin{equation*} 
v(T)=\sup\{|x^*(T(x))|: (x,x^*)\in \Pi(X)\} \ \ \mbox{and} \ \ 
n(X)=\inf\{v(T):  T\in \mathcal{L}(X),\,\, \|T\|=1\}.
\end{equation*} 
It is clear that $0\leq n(X)\leq 1$ and $n(X)\|T\|\leq v(T)\leq \|T\|$ for all $T\in \mathcal{L}(X)$. So, if $n(X) = 1$, then $\|T\| = v(T)$ for every operator $T \in \mathcal{L}(X)$ and we are using this fact throughout the paper without any explicit reference. Moreover, $v(\cdot)$ is a seminorm in $\mathcal{L}(X)$ and if $n(X)>0$, then it becomes an equivalent norm to the usual one on $\mathcal{L}(X)$. With this notation in mind, we say that $T \in \mathcal{L}(X)$ {\it attains the numerical radius} if there is $(x_0, x_0^*) \in \Pi(X)$ such that $|x_0^*(T(x_0))| = v(T)$. We denote by $\NRA(X)$ the set of all numerical radius attaining operators on $X$. We refer the interested reader in this topic to the classical books \cite{BD1,BD2}.

We say that $X$ satisfies the {\it Bishop-Phelps-Bollob\'as property for numerical radius} (BPBp-nu, for short) if given $\e > 0$, there exists $\eta(\eps)>0$ such that
whenever $T\in \mathcal{L}(X)$ with $v(T) = 1$ and $(x, x^*)\in \Pi(X)$ satisfy 
\begin{equation*} 
|x^*(T(x))|>1-\eta(\eps),
\end{equation*} 
there are $S\in \mathcal{L}(X)$ with $v(S) = 1$ and $(y, y^*)\in \Pi(X)$ such that
\begin{equation*} 
|y^*(S(y))|=1, \ \ \  \|x-y\|<\eps, \ \ \ \|x^*- y^*\|<\eps \ \ \ \mbox{and} \ \ \ \|S - T\| < \e.
\end{equation*} 
It is immediate to see that if $X$ satisfies the BPBp-nu, then the set of all operators that attain the numerical radius is norm dense in $\mathcal{L}(X)$. As an overview of the known results, a Banach space $X$ satisfies the BPBp-nu when
\begin{itemize}
\item $X$ is finite dimensional (see \cite[Proposition 2]{KLM});
\item $X$ is uniformly convex and uniformly smooth, $n(X)>0$ (see \cite[Proposition 4 and 6]{KLM});
\item $X$ is Hilbertian  (see \cite[Corollary 3.3]{KLMM});
\item $X$ is $c_0$ or $\ell_1$  (see \cite{GK});
\item $X$ is $L_1(\mu)$ (see \cite[Theorem 9]{KLM} and also \cite[Theorem 9]{F});
\item $X$ is the subspace of all
\begin{itemize}
	\item[$\circ$] finite-rank operators on $L_1(\mu)$;
	\item[$\circ$]  compact operators on $L_1(\mu)$;
	\item[$\circ$]  weakly compact operators on $L_1(\mu)$,
\end{itemize}
where $\mu$ is $\sigma$-finite measure space (see \cite[Theorem 2.1 and Corollary 2.1]{AFS};
\item $X$ is $C(K)$ with compact metrizable $K$ (real) (see \cite[Theorem 2.2]{AGR}).
\end{itemize}

Very recently, a stronger property than the BPBp-nu was considered in \cite[Theorem 2.5]{CDJ}. Stronger in the sense that if we have $T \in \mathcal{L}(X)$ with $v(T) = 1$ and $(x, x^*) \in \Pi(X)$ satisfying $|x^*(T(x))|\approx 1$, then the new operator $S \in \mathcal{L}(X)$ with $v(S) = 1$ will satisfy $|x^*(S(x))| = 1$ and $S \approx T$. That is, we do not change the initial state $(x, x^*)$ where $T$ almost attains the numerical radius.
As occurs with the BPBp-nu, which is a  numerical radius version of the Bishop-Phelps-Bollob\'as property for operators defined in \cite{AAGM}, this new property is the corresponding of the Bishop-Phelps-Bollob\'as {\it point} property (see \cite{DKL} and the references therein) for the numerical radius.
This is one of the properties we are focusing on in this paper. Analogously, we consider a property that instead of fixing the state $(x,x^*)$, we fix the operator $T$. This is the corresponding version of the Bishop-Phelps-Bollob\'as {\it operator} property (see \cite{DKKLM} and the references therein) for the numerical radius. We are giving the precise definitions throughout the next sections.

Let us now describe the content  of this paper. In the next section, we consider both Bishop-Phelps-Bollob\'as point and operator properties for the numerical radius. In Theorem~\ref{hilbert}, we prove that a space with micro-transitive norm and second numerical index strictly positive satisfy the point property for the numerical radius.  In particular, real Hilbert spaces satisfy this property, a result that generalize \cite[Theorem 2.5]{CDJ}. We also show that, for Banach spaces with numerical index 1, both point and operator properties for the numerical radius are too strong, in the sense that just one-dimensional spaces enjoy it (see Proposition~\ref{index_1-1}). In Proposition~\ref{index_1-1-2} we focus on the operator property for the numerical radius showing that, indeed, it is a very restrictive property: under some very general assumptions on the space $X$,  the operator property holds if and only if $X$ is one-dimensional. 
In Section \ref{local}, we consider the corresponding local versions of the already mentioned point and operator properties for the numerical radius, meaning that the $\eta$ that appears in their definitions depends not only on $\e$, but also on a state $(x, x^*) \in \Pi(X)$ or on a numerical radius one operator $T$. In Proposition~\ref{index_1-3} we prove that the local point property for the numerical radius implies that the Banach space must have strongly subdifferentiable norm, whenever its numerical index is 1. We also prove that finite dimensional spaces with $n(X) > 0$ satisfy it. In particular, every complex finite dimensional Banach space satisfy this property. Moreover, in Theorem~\ref{c_0} we prove that the Banach space $c_0$ has the local point property for the numerical radius (whereas, on the other hand, $\ell_1$ fails it). As the strong subdifferentiability of the norm is very related to this property, we also ask if it is a sufficient condition for its validity: in Theorem~\ref{counterexample} we see that this is not the case, by exhibiting a counterexample. 
 We finish the paper by considering the local operator property for the numerical radius. In Theorem~\ref{fdim} and Proposition~\ref{Loo finite dim} we prove that every finite dimensional space enjoys this property and that, for spaces $X$ with the approximation property and $n(X) > 0$, the local operator property is, indeed, equivalent to finite dimensionality of the space.

\section{The point and operator properties for numerical radius}

In this section, we study both {\it uniform} Bishop-Phelps-Bollob\'as point and operator properties for numerical radius.  By uniform, we mean that the $\eta$ that appears in their definitions depends just on a given $\e > 0$ (in contrast with the {\it local} properties defined in Section~\ref{local}, where the $\eta$ depends on $\e$ and a state, or $\e$ and an operator). It is worth noting that the Bishop-Phelps-Bollob\'as point property for numerical radius was already introduced by Choi et al. in \cite{CDJ} in the context of complex Hilbert spaces.

\begin{mydef} \label{def1} Let $X$ be a Banach space. We say that $X$ has the
\begin{enumerate}
\item[\rm (i)]  {\it Bishop-Phelps-Bollob\'as point property for numerical radius} (BPBpp-nu, for short) if given $\e > 0$, there is $\eta(\e) > 0$ such that whenever $T \in \mathcal{L}(X)$ with $v(T) = 1$ and $(x, x^*) \in \Pi(X)$ satisfy $|x^*(T(x))| > 1 - \eta(\e)$, there exists $S \in \mathcal{L}(X)$ with $v(S) = 1$ such that $|x^*(S(x))| = 1$ and $\|S - T\| <\e$.

\vspace{0.2cm}
	
\item[\rm (ii)]  {\it Bishop-Phelps-Bollob\'as operator property for numerical radius} (BPBop-nu, for short) if given $\e>0$, there is $\eta(\e) > 0$ such that whenever $T \in \mathcal{L}(X)$ with $v(T) = 1$ and $(x, x^*) \in \Pi(X)$ satisfy $|x^*(T(x))| > 1 - \eta(\e)$, there is $(y, y^*) \in \Pi(X)$ such that $|y^*(T(y))| = 1$, $\|y - x\| < \e$, and $\|y^* - x^*\| < \e$.
\end{enumerate}
\end{mydef}

First, we focus on the point property.  We prove that every Banach space with micro-transitive norm and second numerical index strictly positive satisfy the {\bf BPBpp-nu} and, also, that this property is too strict when we consider Banach spaces with numerical index 1. In order to do this, we need some background.

Given a Banach space $X$,  an operator $T \in \mathcal{L}(X)$ is said to be {\it skew-hermitian} if $v(T) = 0$. We denote by $\mathcal{Z}(X)$ to the set of all skew-hermitian operators on $X$, which is a closed subspace of $\mathcal{L}(X)$. In the quotient space $\mathcal{L}(X) / \mathcal{Z}(X)$, we define $\|T + \mathcal{Z}(X)\| := \inf \{\|T - S\|: S \in \mathcal{Z}(X)\}$.  Then, we have that $v(T) \leq \|T + \mathcal{Z}(X)\|$ for every $T \in \mathcal{L}(X)$. The \emph{second numerical index} of  $X$  is defined by
\begin{eqnarray*}
	n'(X) &=& \max ~ \{k \geq 0~:~ k \|T + \mathcal{Z}(X)\| \leq v(T) \ \forall \ T \in \mathcal{L}(X) \}\\
	&=&\text{inf~} \left\{\frac{v(T)}{\|T+\mathcal{Z}(X)\|}~:~T \in \mathcal{L}(X)\setminus\mathcal{Z}(X)\right\} 
\end{eqnarray*}
We refer the interested reader on this topic to \cite{KLMM}, where many properties on the second numerical index were obtained and the condition $n'(X)=1$ was intensively studied.

Let $G$ be a Hausdorff topological group with the identity element $e$ and $T$ be a Hausdorff topological space. An \emph{action} $(\cdot.\cdot)$ of $(G,T)$ is a continuous function from $G\times T$ to $T$ such that $(e,t)=t$ and  $(g_1,(g_2,t))=(g_1 g_2,t)$ for every $g_1,g_2\in G$ and $t\in T$. The action is said to be \emph{transitive} if $T=\{(g,t)\colon g\in G\}$ for every $t\in T$, and said to be \emph{micro-transitive} if $\{(g,t)\colon g\in U\}$ is a neighborhood of $t$ in $T$ for every $t\in T$, whenever $U$ is a neighborhood of $e$ in $G$. Given a Banach space $X$, we may take the group of surjective isometries on $X$ as the group $G$ and $S_X$ as the topological space $T$. We then say that $X$ (or the norm of $X$) is \emph{micro-transitive} (respectively, \emph{transitive}) if the canonical action is micro-transitive (respectively, transitive). It is known that micro-transitivity of a norm implies transitivity. The famous open problem, known as the Mazur rotation problem, asks whether transitive separable Banach spaces are isometrically isomorphic to Hilbert spaces (see, for instance, \cite{Banach}). It is worth remarking that the non-separable version of the Mazur rotation problem had been solved negatively by Rolewicz (see \cite{Rolewicz}) and Hilbert spaces are the only known spaces with micro-transitive norms. We kindly send the interested reader on this topic to \cite{BecerraRodriguez2002,Effros} and the references therein.

Our main result in this section is the following.

\begin{theorem} \label{hilbert} Let $X$ be a Banach space. Suppose that the norm of $X$ is micro-transitive and that $n'(X) > 0$. Then, $X$ satisfies the BPBpp-nu.
\end{theorem}

Let us notice that if $X$ has a micro-transitive norm, then there is a function $\beta: (0, 2) \longrightarrow \R^+$ such that whenever $x, y \in S_X$ satisfy $\|x - y\| < \beta(\e)$, there is a surjetive isometry $T \in \mathcal{L}(X)$ satisfying $T(x) = y$ and $\|T - \Id_X\| < \e$, where $\Id_X$ is the identity operator on $X$ (see \cite[Proposition 2.1]{CDKKLM}). It is worth remarking that we may take $\beta$ so that $\beta(\e)<\e$ for any $\e \in (0, 2)$.

\begin{proof}[Proof of Theorem \ref{hilbert}] Suppose that $X$ is micro-transitive with some function $\e \mapsto \beta(\e)$. Then, by \cite[Corollary 2.13]{CDKKLM}, we see that $X$ is both uniformly convex and uniformly smooth. Moreover since $n'(X) > 0$, $X$ satisfies the BPBp-nu with some function $\e \mapsto \eta(\e)$ (see \cite[Theorem 3.2]{KLMM} and \cite[Proposition~4]{KLM}).

 Let $\e \in (0, 1)$ be given and set 
\begin{equation*} 
\e':=\frac{n'(X)}{2 +5 n'(X)}\e > 0\ \ \ \mbox{and} \ \ \ \eta'(\e') := \eta(\beta(\e')) > 0.
\end{equation*} 
Let $(x, x^*) \in \Pi(X)$ and $T \in \mathcal{L}(X)$ with $v(T) = 1$ be such that $|x^*( T(x) ) | > 1 - \eta'(\e')$. From the definition of the second numerical index, there exists $G\in \mathcal{Z}(X)$ so that
\begin{equation*} 
\|T+G\|<\frac{v(T)}{n'(X)}+\e'=\frac{1}{n'(X)}+\e'.
\end{equation*} 
Since $v(G) = 0$, we also have that $v(T+G)=1$ and $|x^*( (T+G)(x) ) | = |x^*(T(x))| > 1 - \eta'(\e')$. Therefore, there are $(y, y^*) \in \Pi(X)$ and $S_1 \in \mathcal{L}(X)$ with $v(S_1) = 1$ such that
\begin{equation*}
|y^*(S_1 (y))| = 1 \ \ \ \mbox{and} \ \ \  \max \left\{\|y^* - x^*\|, \|y - x\|, \|S_1 - (T+G)\| \right\} < \beta(\e')<\e'.
\end{equation*}
In particular $\|y - x\| < \beta(\e')$ and, since $X$ is micro-transitive with the function $\beta$, there is a linear surjective isometry $U \in \mathcal{L}(X)$ such that $U(x) = y$ and $\|U - \Id_X\| < \e'$. Now, notice that $1 = y^*(y) = y^*( U(x) ) = (U^*(y^*)) (x) $. Since $X$ is uniformly smooth and $x^*(x) = 1$, we have that $U^* (y^*) = x^*$ Analogously, we get that $(U^{-1})^* (x^*) = y^*$.

Define $S_2 := U^{-1} \circ S_1 \circ U \in \mathcal{L}(X)$. Since $\left((U^{-1})^*(z^*)\right)(U(z) ) = z^*(U^{-1} U (z))  =  z^*(z) = 1$, whenever $(z, z^*) \in \Pi(X)$, we have that $v(S_2) \leq v(S_1) = 1$. Moreover, the equality 
\begin{equation*}
| x^*(S_2 (x) )| = | (x^*( U^{-1} S_1 U (x) ) | \\
= | ( (U^{-1})^* (x^*))( S_1 U(x) ) | 
=| y^*( S_1 (y) )| 
= 1
\end{equation*}
shows that $S_2$ has numerical radius $1$ and it is attained at $(x, x^*) \in \Pi(X)$. Consequently, the operator $S:= S_2-G$ also has numerical radius $1$ and it is attained at $(x, x^*)$. Finally,  
\begin{eqnarray*}
\|(S_2 - G) - T \| &\leq& \|S_2-S_1\|+\|S_1-(T+G)\|\\
&=& \|(U^{-1}S_1U -S_1U\|+\|S_1U-S_1\| + \e' \\
&<& 2 \e' \|S_1\| + \e' \\
&<& 2 \e' \left(\frac{1}{n'(X)} +2\e'\right) + \e' \\
&=& \left(\frac{2 +(4\e'+1) n'(X)}{n'(X)} \right) \e'\leq \e.
\end{eqnarray*}
\end{proof}

Since real Hilbert spaces have second numerical index $1$ (see \cite[Theorem 2.3]{KLMM}), we have that real Hilbert spaces satisfy the BPBpp-nu. On the other hand, since the numerical index of a complex Banach space $X$ is always greater than or equal to $1/e$ (and, hence, strictly positive), we have $\mathcal{Z}(X) = \{0\}$ and, consequently, $n'(X) = n(X) > 0$. This means that we have the same result for complex Hilbert spaces. We state these results in the next corollary, which should be compared with \cite[Theorem~4.1]{CDJ}.(a).

\begin{cor} \label{hilbert1} Let $H$ be a (real or complex) Hilbert space. Then, $H$ has the  BPBpp-nu.	
\end{cor}

In what follows, we focus on spaces with numerical index 1. Examples of such spaces include $C(K)$-spaces, $L_1(\mu)$-spaces, isometric preduals of $L_1(\mu)$, and all function algebras, such as the disk algebra $A(\mathbb{D})$ and $H^{\infty}$. We prove that these spaces satisfy neither the BPBpp-nu nor the  BPBop-nu.
To prove this, we need the following characterizations of uniformly smooth and uniformly convex Banach spaces. On the one hand, we have that $X$ is uniformly smooth if and only if given $\e>0$, there is $\eta(\e) > 0$ such that whenever $x \in S_X$ and $x^* \in S_{X^*}$ satisfy $|x^*(x)| > 1 - \eta(\e)$, there is $x_0^* \in S_{X^*}$ such that $|x_0^*(x)| = 1$ and $\|x_0^* - x^*\| < \e$ (see \cite[Proposition 2.1]{DKL}). On the other hand, it is proved in \cite[Theorem~2.1]{KL} that $X$ is uniformly convex if and only if given $\e > 0$, there is $\eta(\e) > 0$ such that whenever $x \in S_X$ and $x^* \in S_{X^*}$ satisfy $|x^*(x)| > 1 - \eta(\e)$, there is $x_0 \in S_X$ such that $|x^*(x_0)| = 1$ and $\|x_0 - x\| < \e$.

\begin{prop} \label{index_1-1} Let $X$ be a Banach space with $n(X) = 1$. The following are equivalent.

\begin{itemize}
\item[(a)] $X$ has the BPBpp-nu.
\item[(b)] $X$ has the BPBop-nu.
\item[(c)] $X$ is one-dimensional.
\end{itemize}
\end{prop}

\begin{proof} Let us prove (a)$\Leftrightarrow$(c). In order to do this, we only need to show that if $X$ satisfy the BPBpp-nu then it is one-dimensional, since the converse is trivial. Assume that the following holds.

\noindent\textbf{Claim:} if $X$ has the BPBpp-nu and $n(X)=1$, then $X$ is uniformly smooth.

\noindent In that case, we have that the dual space $X^*$ is uniformly convex and satisfies the alternative Daugavet property (since $n(X)=1$). By \cite[Theorem 2.1]{KMMP}, $X$ must be one-dimensional.

To prove the claim above, suppose that $X$ has the BPBpp-nu with some function $\e \mapsto \eta(\e)$. Let $\e \in (0, 1)$ be given and let $x_0^* \in S_{X^*}$ and $x_0 \in S_X$ be such that $|x_0^*(x_0)| > 1 - \eta(\e)$. Let $x_1^* \in S_{X^*}$ be such that $x_1^*(x_0) = 1$ and consider the operator $T := x_0^* \otimes x_0$. By hypothesis we have that $v(T) = \|T\| = 1$ and, since $|x_1^*(T(x_0))| = |x_1^*(x_0^*(x_0)x_0)| = |x_0^*(x_0)| > 1 - \eta(\e)$, that there is $S \in \mathcal{L}(X)$ with $v(S) = \|S\| = 1$ (again, by the hypothesis $n(X)=1$) such that $|x_1^*(S(x_0))| = 1$ and $\|S - T\| < \e$. Define $z_0^* := S^*x_1^* \in X^*$. Then, we have that $\|z_0^*\| \leq 1$ and $|z_0^*(x_0)| = 1$. This shows, in particular, that $\|z_0^*\| = 1$. Also, for all $x \in S_X$,
	\begin{eqnarray*}
		|z_0^*(x) - x_0^*(x)| &=& |x_1^*(S(x)) - x_0^*(x) x_1^*(x_0)| \\
		&=& |x_1^*(S(x)) - x_1^*(x_0^*(x)x_0)| \\
		&\leq& \|S(x) - T(x)\| 
	\end{eqnarray*}
Thus, $\|z_0^* - x_0^*\|\leq \|S-T\|< \e$ and so $X$ is uniformly smooth. 

Since the equivalence (b)$\Leftrightarrow$(c) follows analogously, we omit the proof.
\end{proof}

In particular, we have that $\ell_1$ and $c_0$ are examples of Banach spaces with numerical index 1 which satisfy the BPBp-nu but fail both BPBpp-nu and BPBop-nu by Proposition \ref{index_1-1}. In the next section, we show that $c_0$ satisfies a local point property, but still fail the local operator one (see Theorem~\ref{c_0} and Proposition~\ref{Loo finite dim}, respectively). As a matter of fact, the situation for the uniform and local operator properties seem to be very restrictive: even 2-dimensional spaces does not enjoy it. To see that, let $X$ be a 2-dimensional Banach space and consider $\{(e_1, e_1^*), (e_2, e_2^*)\}$ the Auerbach basis of the space $X$. Then, $x=e_1^*(x)e_1+e_2^*(x)e_2$ for every $x\in X$. Put $\beta_n=1-\frac{1}{n}$ and define the operator $T_n\colon X\to X$ by $T_n(x)=\beta_n e_1^*(x)+e_2^*(x)e_2$. It is easy to see that  $v(T_n)=1$ and $|e_1^*(T_n(e_1))|=\beta_n=1-\frac{1}{n}$. However, $T_n$ does not attain the numerical radius at any state $(y,y^*) \in \Pi(X)$ close to $(e_1, e_1^*)$. Indeed, if it was the case, we would have
$$
1\leq |y^*(T_n(y))|\leq \|T_n(y)\|\leq \beta_n + (1-\beta_n) |e_2^*(y)|\leq 1
$$ 
and, consequently, $|e_2^*(y)|=1$, which clearly implies $\|e_1-y\|\geq 1$. It is worth noting that a similar argument works for $n$-dimensional spaces with $n\geq 2$. Then, any finite dimensional Banach space of dimension greater than one fails the BPBop-nu.

Moreover, noting that the BPBop-nu trivially implies its local version and taking Proposition~\ref{Loo finite dim} below into account, we see that if $n(X)>0$, $X$ has the approximation property and $X$ has the BPBop-nu, then $X$ must be one-dimensional. Hence, we have the following result.

 \begin{theorem} \label{index_1-1-2} Let $X$ be a Banach space  which satisfies one of the following conditions.
 \begin{itemize}
\item[(a)] $n(X) = 1$
\item[(b)] $n(X)>0$ and $X$ has the approximation property
\item[(c)] $X$ is finite dimensional
\end{itemize}
Then, $X$ has the BPBop-nu if and only if it is one-dimensional.
\end{theorem}

We finish this section by asking some questions.

\begin{question} Let $X$ be a Banach space.
\begin{itemize}
	\item[(1)] If $X$ has the BPBpp-nu and $0<n(X)<1$, then $X$ is uniformly smooth? 
	\item[(2)] If $X$ has the BPBop-nu and $0<n(X)<1$, then $X$ is one-dimensional?
\end{itemize}
\end{question}

\section{Local Properties} \label{local} 

Motivated by the very restrictive behavior of the uniform versions (as we have seen in the previous section), we are now dealing with weaker properties: we consider the local versions of both BPBpp-nu and BPBop-nu. By weaker, we mean that the $\eta$ that appears in their definitions does not depend just on a given $\e > 0$ but also on a state $(x, x^*) \in \Pi(X)$ or on a numerical radius one operator $T \in \mathcal{L}(X)$. For the Bishop-Phelps-Bollob\'as properties, these local properties were already considered in \cite{DKLM, DKLM2}. Here, we keep a similar notation.

\begin{mydef} \label{def2} Let $X$ be a Banach space. We say that $X$ has the 
\begin{itemize}
\item[(i)]
{\bf L}$_{p,p}$-nu if given $\e > 0$ and $(x, x^*) \in \Pi(X)$, there is $\eta(\e, (x, x^*)) > 0$ such that whenever $T \in \mathcal{L}(X)$ with $v(T) = 1$ satisfies $|x^*(T(x))| > 1 - \eta(\e, (x, x^*))$,	there is $S \in \mathcal{L}(X)$ with $v(S) = 1$ such that $|x^*(S(x))| = 1$ and $\|S - T\| <\e$.

\vspace{0.2cm}	 

\item[(ii)] {\bf L}$_{o,o}$-nu if given $\e > 0$ and $T\in \mathcal{L}(X)$ with $v(T)=1$, there is $\eta(\e, T) > 0$ such that whenever $(x, x^*) \in \Pi(X)$ satisfies $|x^*(T(x))| > 1 - \eta(\e, T)$,
there is $(y, y^*) \in \Pi(X)$ such that $|y^*(T(y))| = 1$, $\|y - x\| <\e$, and $\|y^* - x^*\| <\e$. 
	\end{itemize}
\end{mydef}

\noindent 
It is immediate to notice that the BPBpp-nu and the BPBop-nu imply properties {\bf L}$_{p,p}$-nu and {\bf L}$_{o,o}$-nu, respectively, and also that if $X$ satisfies the {\bf L}$_{o,o}$-nu, every operator attains its numerical radius.

We start by proving that all spaces with numerical index 1 which satisfy the {\bf L}$_{p,p}$-nu must be strongly subdifferentiable. Let us recall that a norm in a Banach space $X$ is {\it strongly subdifferentiable} (SSD, for short) at $x \in X$ whenever the limit $\lim_{t \rightarrow 0^+} \frac{1}{t}(\|x + th\| - \|x\|)$ exists uniformly for $h \in B_X$. The norm of any finite dimensional space and the sup-norm on $c_0$ are examples of SSD norms. Moreover, the $\ell_1$-norm is SSD only at points in the sphere of $\ell_1$ that are sequences with finitely many nonzero terms. For a background on this topic, we refer the reader to \cite{FP} and the references therein.  

\begin{prop} \label{index_1-3} Let $X$ be a Banach space with $n(X) = 1$. If $X$ has the {\bf L}$_{p,p}$-nu, then the norm of $X$ is SSD.	
\end{prop}
\begin{proof} In order to prove that $X$ is SSD, we use a characterization given in \cite[Theorem 2.3.(a)]{DKLM}, which says that the norm of $X$ is SSD at $x \in S_X$ if and only if given $\e>0$, there is $\eta(\e, x) > 0$ such that whenever $x^* \in S_{X^*}$ satisfies $|x^*(x)| > 1 - \eta(\e)$, there is $z^* \in S_{X^*}$ such that $|z^*(x)| = 1$ and $\|z^* - x^*\| < \e$.  Let $\e > 0$ and $x_0 \in S_X$ be given.

  Let $x_0^* \in S_{X^*}$ be such that $x_0^*(x_0) = 1$, which implies $(x_0, x_0^*) \in \Pi(X)$. Since $X$ has the {\bf L}$_{p,p}$-nu, we may consider $\eta(\e, x_0) := \eta(\e, (x_0, x_0^*)) > 0$. Let $x^* \in S_{X^*}$ be a functional such that $|x^*(x_0)| > 1 - \eta(\e, x_0)$ and define $T := x^* \otimes x_0$. Then $v(T) = \|T\| = \|x^*\| = 1$ and $|x_0^*(T(x_0))| = |x^*(x_0)| > 1 - \eta(\e, x_0)$. By hypothesis, there is $S \in \mathcal{L}(X)$ with $v(S) = \|S\| = 1$ such that $|x_0^*(S(x_0))| = 1$ and $\|S - T\| < \e$. Setting $z^* := S^* x_0^* \in B_{X^*}$, we have that $|z^*(x_0)| = |(S^* (x_0^*)) (x_0)| = |x_0^*(S(x_0))| = 1$.  Moreover, we get that $\|z^* - x^*\| < \e$. Indeed, for arbitrary $x \in S_X$,
	\begin{eqnarray*}
		|z^*(x) - x^*(x)| &=& |(S^* (x_0^*))(x) - x^*(x)| \\
		&=& |x_0^*(S(x)) - x^*(x)x_0^*(x_0)| \\
		&=& |x_0^*(S(x)) - x_0^*(x^*(x)x_0)| \\
		&=& |x_0^*(S(x)) - x_0^*(T(x))| \\
		&\leq& \|S(x) - T(x)\|.
	\end{eqnarray*}	
This implies that $\|z^*-x^*\| \leq \|S - T\| < \e$ which shows $X$ is SSD at $x_0 \in S_X$. Since $x_0$ is arbitrary, the norm of $X$ is SSD.
\end{proof} 

On the other hand, we do not know what happens in the general case.

\begin{question} If $X$ has the {\bf L}$_{p,p}$-nu and $n(X) \not= 1$, then $X$ must be SSD?
\end{question}

Let us show now some positive results regarding the validity of {\bf L}$_{p,p}$-nu. Thanks to Proposition~\ref{index_1-3}, if we are looking for spaces satisfying {\bf L}$_{p,p}$-nu, it is natural to look at those with strong subdifferentiable norm, even in case that $0<n(X)<1$.
Since finite dimensional spaces are SSD, we analyze the {\bf L}$_{p,p}$-nu for these spaces. Suppose that $X$ is finite dimensional and does not have property {\bf L}$_{p,p}$-nu. Then, by definition, there are $\e_0 > 0$ and  $(x_0, x_0^*) \in \Pi(X)$ such that there exists a sequence of operators $(T_n)  \subset \mathcal{L}(X)$ so that  
\begin{equation*} 
v(T_n) = 1 \geq |x_0^*(T_n(x_0))| \geq 1 - \frac{1}{n} 
\end{equation*} 
for all $n \in \N$ and, whenever $S \in \mathcal{L}(X)$ satisfies $v(S)=1$ and $\|T_n-S\| < \e$, the number $|x_0^*(S(x_0))|$ is {\it strictly} smaller than 1. If the set of operators with numerical radius $1$ is compact, then a subsequence of $(T_n)$ converges to an operator $T_0$, also with numerical radius $1$, and, in this case, we have $|x_0^*( T_0(x_0))| = 1$, which is a contradiction. It is clear that if $n(X)>0$, then $v(\cdot)$ is an equivalent norm on $\mathcal{L}(X)$. Hence, if $X$ is finite dimensional and $n(X)>0$, the closed unit ball of $(\mathcal{L}(X),v(\cdot))$ is compact. So, we have the following result.

\begin{prop} \label{finite} Let $X$ be a finite dimensional Banach space with $n(X) > 0$. Then, $X$ satisfies property {\bf L}$_{p,p}$-nu.
\end{prop}

Arbitrary complex Banach spaces always have strictly positive numerical index. So, we have the following consequence from Theorem \ref{finite}.

\begin{cor} Every finite dimensional complex Banach space has {\bf L}$_{p,p}$-nu.
\end{cor}

We do not know if the same statement holds for real spaces. However, we can conclude that it is true for 2-dimensional real Banach spaces. Indeed, if $X$ is 2-dimensional and $n(X)=0$, then $X$ is isometrically isomorphic to the 2-dimensional Hilbert space (see \cite[Theorem~3.1]{R}) and, consequently, it has the BPBpp-nu.

\begin{question} Finite dimensional spaces with $n(X)=0$ have the {\bf L}$_{p,p}$-nu?
\end{question}

Next, we consider the property on $c_0$ which is SSD and has numerical index 1. It is known that the pairs $(c_0, c_0)$ and $(c_0, X)$, where $X$ is $\C$-uniformly convex, satisfy the {\bf L}$_{p,p}$ (see \cite[Proposition 2.8 and Theorem 2.12]{DKLM}, respectively). In particular, the pair $(c_0, L_p(\mu))$ satisfies it for a positive measure $\mu$ and $1 \leq p < \infty$. Here, we have the following result.

\begin{theorem} \label{c_0} $c_0$ satisfies the {\bf L}$_{p,p}$-nu.	
\end{theorem}

In order to prove this theorem, we need two auxiliary results. The first one concerns a characterization of the strong subdifferentiability of the norm of a Banach space in terms of finite convex  sums, which was proved in \cite{DKLM2}. The second is a straightforward fact about functionals on $c_0$ which attain the norm and we include a proof for completeness.

\begin{lemma}(\cite[Proposition 3.2]{DKLM2}) \label{lemma1} For every Banach space $Y$, the following are equivalent.
\begin{itemize}
	\item[(a)] The norm of $Y$ is SSD.
	\vspace{0.1cm}
	\item[(b)]  For each $\e > 0$, $y \in S_Y$ and a finite sequence $(\alpha_j)_{j=1}^n$ such that $n>0$, $\alpha_j > 0$ for all $j$ and $\sum_{j=1}^n \alpha_j = 1$,  there exists $\eta = \eta(\e, (\alpha_j)_{j \in A}, y) > 0$ such that  whenever a sequence of functionals $(y_j^*)_{j \in A} \subset B_{Y^*}$ satisfies
	\begin{equation*}
		\re \sum_{j \in A} \alpha_j y_j^*(y) > 1 - \eta,	
	\end{equation*}
	there is $(z_j^*)_{j=1}^n \subset S_{Y^*}$ such that
	\begin{equation*}
		z_j^*(y) = 1 \ \ \ \mbox{and} \ \ \ \|z_j^* - y_j^*\| < \e	 \text{~for~all~} j=1,2,...,n.
	\end{equation*}	
\end{itemize}
\end{lemma}

\begin{lemma} \label{lemma2} Let $x^* = (x_i^*)_{i=1}^{\infty} \in S_{c_0^*}$ be a linear functional on $c_0$ and suppose that it attains the norm at $x_0 = (x(i))_{i=1}^{\infty} \in S_{c_0}$. If there is $j \in \N$ such that $0 < |x_0(j)| < 1$, then $x^*(e_j) = 0$, where $(e_i)_{i=1}^{\infty}$ is the canonical basis of $c_0$.	
\end{lemma}

\begin{proof} Suppose that $0 < |x_0(j)| < 1$ and that $x^*(e_j) = x_j^* \not= 0$. Then,
\begin{equation*}
1 = \|x^*\|_1 = |x^*(x_0)| = \left| \sum_i x^*_i x_0(i) \right| <  \sum_{i \not= j} |x_i^*| + |x_j^*| = \|x^*\|_1 = 1,	
\end{equation*}
which is a contradiction.
\end{proof}

\begin{proof} [Proof of Theorem \ref{c_0}]

Denote the canonical basis of $c_0$ and $\ell_1$  by $(e_i)$ and $(e_i^*)$, respectively. Let $(x_0, x_0^*) \in \Pi(c_0)$ be given. We write $x_0 = (x_0(i))_{i=1}^{\infty}=\sum_{i=1}^{\infty} x_0(i)e_i\in S_{c_0}$. Since $x_0 \in S_{c_0}$, there is a finite collection $A=\{n_1, \ldots, n_m\} \subset \N$ such that $|x_0(n_i)|=1$ for $i=1, \ldots, m$. By Lemma \ref{lemma2}, we have 
\begin{equation*}
x_0^* = \alpha_1 e_{n_1}^* + \cdots + \alpha_{m} e_{n_m}^* \ \ \mbox{with} \ \ \|x_0^*\|= \sum_{i=1}^m |\alpha_i| = 1.	
\end{equation*}
We may suppose that all $\alpha_i \not= 0$ and notice that
\begin{eqnarray*}
|\alpha_1| |x_0(n_1)| + \cdots + |\alpha_m| |x_0(n_m)| &=& |\alpha_1| + \cdots + |\alpha_m| \\
&=& 1 \\
&=& x_0^*(x_0) \\
&=& \alpha_1 x_0(n_1) + \cdots + \alpha_m x_0 (n_m)	
\end{eqnarray*}
This implies that $\alpha_i x_0 (n_i) = |\alpha_i x_0 (n_i)|$ for every $i = 1, \ldots, m$. Let $\e>0$ and $0<\xi < \frac{\e}{4}$ be given. Since $c_0$ is SSD, using $\eta$ of Lemma \ref{lemma1}, we consider 
\begin{equation*}
\eta := \eta(\xi, (|\alpha_i|)_{i=1}^m, x_0) > 0.
\end{equation*}
Let $T \in \mathcal{L}(c_0)$ with $v(T) = \|T\| = 1$ be such that $|x_0^*(T(x_0))| > 1 - \eta$. For a suitable modulus $1$ scalar $r$ and $y_{n_i}^* = \frac{\alpha_i}{|\alpha_i|} r e_{n_i}^* \ \mbox{for} \ i = 1, \ldots, m$, we have
\begin{equation*}
|\alpha_1| y_{n_1}^*(T(x_0)) + \cdots + |\alpha_m| y_{n_m}^*(T(x_0))=|x_0^*(T(x_0))|>1-\eta
\end{equation*}
Then, by Lemma \ref{lemma1}, there is $(z_{n_i}^*)_{i=1}^m \subset S_{c_0^*}$ such that $z_{n_i}^*(x_0) = 1$ and $\|z_{n_i}^* - y_{n_i}^*\circ T \| < \xi$ for $i = 1, \ldots, m$. Now, define $S \in \mathcal{L}(c_0)$ by
\begin{equation*}
S(x) := T(x) + \sum_{i=1}^m \left[ \left( 1 + \frac{\e}{4} \right) z_{n_i}^*(x) - y_{n_i}^*(T(x)) \right] y_{n_i}
\end{equation*}
where $y_{n_i} := \frac{\overline{\alpha_i}}{|\alpha_i|}\overline{r} e_{n_i} \ \mbox{for} \ i=1,\ldots, m$ and $x \in c_0$.
Then, it is clear that
\begin{equation*}
S^*(y^*) = T^*(y^*) + \sum_{i=1}^m y^*(y_{n_i}) \left[ \left( 1 + \frac{\e}{4} \right) z_{n_i}^* - T^*(y_{n_i}^*) \right] \ \ \ (y^* \in c_0^*).
\end{equation*}
 We have that if $n \not= n_1, \ldots, n_m$, then $\|S^*(e_{n}^*)\| = \|T^*(e_n^*)\| \leq 1$. On the other hand, if $n = n_j$ for some $j = 1, \ldots, m$, then
\begin{equation*}
\|S^*(e_{n_j}^*)\| = \|S^*(y_{n_j}^*)\| = \left\| T^*(y_{n_j}^*) + \left( 1 + \frac{\e}{4} \right) z_{n_j}^* - T^*(y_{n_j}^*) \right\| = 1 + \frac{\e}{4}.	
\end{equation*}
Since $v(S) = \|S\| = \|S^*\| = \sup_n \|S^*(e_n^*)\|$, we have that $v(S)=1 + \frac{\e}{4}$. Moreover, we have that \begin{eqnarray*}
|x_0^*(S(x_0))| &=& \left| x_0^*(T(x_0)) + \sum_{i=1}^{m} \left[ \left(1 + \frac{\e}{4} \right) z_{n_i}^*(x_0) - y_{n_i}^*(T(x_0)) \right] x_0^*(y_{n_i}) \right| \\
&=& \left| x_0^*(T(x_0)) + \overline{r}\sum_{i=1}^m |\alpha_i| \left( 1 + \frac{\e}{4} \right) - \overline{r}\sum_{i=1}^m |\alpha_i| y_{n_i}^*(T(x_0)) \right| \\
&=& \left| \sum_{i=1}^m \alpha_i e_{n_i}^*(T(x_0))  + \overline{r}\sum_{i=1}^m |\alpha_i| \left( 1 + \frac{\e}{4} \right) - \sum_{i=1}^m \alpha_i e_{n_i}^*(T(x_0)) \right| \\
&=& 1 + \frac{\e}{4}.
\end{eqnarray*}
Then, $v(S) = \|S\| = |x_0^*(S(x_0))| = 1 + \frac{\e}{4}$. Hence, the inequality
 \begin{equation*}
 	\left\|\frac{S}{1 + \frac{\e}{4}} - T\right\| \leq \left\|\frac{S}{1 + \frac{\e}{4}} - S\right\|+\left\|S - T\right\|<  \frac{\e}{4}+\xi + \frac{\e}{4} < \e
 \end{equation*}
shows that $\ds \frac{S}{1 + \frac{\e}{4}} \in \mathcal{L}(X)$ is the desired operator.
\end{proof}

\begin{rem} \label{l_1} Theorem \ref{c_0} shows that property {\bf L}$_{p,p}$-nu is strictly weaker than the BPBpp-nu since $c_0$ fails to have the BPBpp-nu (see Proposition \ref{index_1-1}). Concerning $\ell_1$, since the numerical index of $\ell_1$ is 1 and $\ell_1$ is {\it not} SSD, it cannot have property {\bf L}$_{p,p}$-nu by Proposition \ref{index_1-3}. Moreover, we can also notice that the denseness of numerical radius attaining operators does not imply property {\bf L}$_{p,p}$-nu since $\overline{\NRA(\ell_1)} = \mathcal{L}(\ell_1)$ (see \cite{C2}) but $\ell_1$ fails property {\bf L}$_{p,p}$-nu. On the other hand, we do not know whether {\bf L}$_{p,p}$-nu implies the denseness of numerical radius attaining operators or not.
	
\begin{question} Let $X$ be a Banach space with the {\bf L}$_{p,p}$-nu. Is it true that $\overline{\NRA(X)} = \mathcal{L}(X)$?
\end{question}	
	 
\end{rem} 

In view of Theorems~\ref{finite} and \ref{c_0}, it is natural to ask whether the strong subdifferentiability of the norm  is a sufficient condition for the validity of property {\bf L}$_{p,p}$-nu. We show that this is not the case by exhibiting the following counterexample. Let $\Z$ be the space $c_0$ equipped with the equivalent strictly convex norm
$$
\|x\|_{\mathcal{Z}}=\|x\|_\infty +\left(\sum_{i=1}^\infty \frac{|x(i)|^2}{2^i} \right)^{1/2}.
$$
This space appears in classical counterexamples for norm attaining results. In \cite{Lind} it is proved that the set $\NA (c_0, \mathcal{Z})$ is not dense in $\mathcal{L}(c_0,\mathcal{Z})$, where $\NA (c_0, \mathcal{Z})$ is the set of all norm attaining operators from $c_0$ to $\mathcal{Z}$. Indeed, if $T \in \NA (c_0, \mathcal{Z})$ then there is $n_0\in \mathbb{N}$ such that $T(e_n)=0$ for all $n\geq n_0$. Hence, the formal identity $id\colon c_0 \to \mathcal{Z}$ cannot be approximated by norm attaining operators. As a consequence, it is also shown in \cite{Lind} that, if we take $X=c_0\oplus_\infty \Z$, then $\NA (X, X)$ is not dense in $\mathcal{L}(X,X)$. When looking at the denseness of numerical radius attaining operators, the space $X=c_0\oplus_\infty \Z$ appears as a natural candidate to show that $\NRA (X, X)$ is not dense in $\mathcal{L}(X,X)$ and, indeed, this is the main result in \cite{P}. Before showing the desired counterexample, we observe two facts.

\begin{remark}\rm
The norm of the Banach space $c_0\oplus_\infty\Z$ is strongly subdifferentiable.
\end{remark}
\begin{proof}
Since $c_0$ is SSD, by \cite[Proposition~2.2]{FP} it suffices to show that $\Z$ is SSD. We want to prove that, for each $z\in S_{\Z}$, the one-side limit 
$
\lim_{t\to 0^+} \frac{\|z+th\|_{\Z}-1}{t}
$
exists uniformly for $h\in B_{\Z}$. Consider $\tilde{z}=\left(\frac{z(i)}{\sqrt{2}^i}\right)$ and $\tilde{h}=\left(\frac{h(i)}{\sqrt{2}^i}\right)$ and note that
$$
\frac{\|z+th\|_{\Z}-1}{t}=\underbrace{\frac{\|z+th\|_{\infty}-\|z\|_\infty}{t}}_{(I_t)} + \underbrace{\frac{\|\tilde{z}+t\tilde{h}\|_{2}-\|\tilde{z}\|_2}{t}}_{(II_t)}.
$$
On the one hand, since $z\in c_0$ and $(c_0, \|\cdot\|_\infty)$ is SSD, we have that $\lim_{t\to 0^+} (I_t)$ exists uniformly for $h\in B_{\Z}$ (note that if $h\in B_{\Z}$, then $h\in B_{c_0}$). On the other hand, since $\tilde{z} \in \ell_2$ and $(\ell_2, \|\cdot\|_2)$ is SSD, then $\lim_{t\to 0^+} (II_t)$ exists uniformly for $h\in B_{\Z}$ (again, note that if $h\in B_{\Z}$, then $\tilde{h}\in B_{\ell_2}$). Then, $\lim_{t\to 0^+} \frac{\|z+th\|_{\Z}-1}{t}$
exists uniformly for $h\in B_{\Z}$.
\end{proof}

\begin{remark}\label{1}\rm
Suppose that a Banach space $X$ has the {\bf L}$_{p,p}$-nu. Then, given $\e>0$ and $(x,x^*)\in \Pi(X)$, there is $\tilde{\eta}(\e, (x,x^*))>0$ such that, whenever $T\in \mathcal{L}(X)$ with $v(T)\leq \|T\|\leq 1$ satisfies
$$
|x^*(T(x))|>1-\tilde{\eta}(\e, (x,x^*)),
$$
there is $S\in \mathcal{L}(X)$ with $v(S)=1$ such that $|x^*(S(x))|=1$ and $\|S-T\|<\e$. Notice that the difference with the property {\bf L}$_{p,p}$-nu is that we are considering the initial operator $T$ in the ball of the space, instead of considering $T$ such that $v(T)=1$. Indeed, let $0<\e<1$ and $0<\eta(\frac{\e}{2}, (x,x^*))<\e/2$ as in the definition of the {\bf L}$_{p,p}$-nu. Put $\tilde{\eta}(\e, (x,x^*))=\frac{\eta(\frac{\e}{2}, (x,x^*))}{2}$ and suppose $T\in \mathcal{L}(X)$ is such that $v(T)\leq \|T\|\leq 1$ and $|x^*(T(x))|>1-\tilde{\eta}(\e, (x,x^*))$. Then, if we consider $T_1=\frac{T}{v(T)}$, we have $v(T_1)=1$ and 
$$
|x^*(T_1(x))|=\frac{1}{v(T)}|x^*(T(x))| >\frac{1-\tilde{\eta}(\e, (x,x^*))}{v(T)}\geq 1-\tilde{\eta}(\e, (x,x^*)).
$$
By hypothesis, there is $S\in \mathcal{L}(X)$ with $v(S)=1$ such that $|x^*(S(x))|=1$ and $\|S-T_1\|<\e/2$. Then, 
\begin{eqnarray*}
\|S-T\|&\leq& \|S-T_1\|+\|T_1-T\| < \frac{\e}{2} + \frac{\|T\|}{v(T)} (1-v(T))\\
&<&  \frac{\e}{2} + \frac{\|T\|}{v(T)} \tilde{\eta}(\e, (x,x^*)) <  \frac{\e}{2} + \frac{\|T\|}{v(T)} \frac{\e}{4}.
\end{eqnarray*}
Since $\|T\|\leq1$ and $v(T)^{-1}\leq 2$ we deduce $\|S-T\|<\e$, which is the desired statement.
\end{remark}

\begin{theorem}\label{counterexample}
The space $X=c_0\oplus_\infty \Z$ is SSD and fails the  {\bf L}$_{p,p}$-nu.
\end{theorem}

\begin{proof}
Assume that $X$ has the {\bf L}$_{p,p}$-nu and fix $z_0 \in S_{\Z}$ such that $z_0(1)\geq z_0(2)\geq \cdots >0$. Note that $z_0(1)\geq \frac{1}{2}$, otherwise, we would have
$$
\|z_0\|_{\Z}=\|z_0\|_\infty+\left(\sum_{i=1}^\infty \frac{|z_0(i)|^2}{2^i} \right)^{1/2} < \frac{1}{2} + \frac{1}{2}\left(\sum_{i=1}^\infty \frac{1}{2^i} \right)^{1/2}=1.
$$
For $z_0^* \in S_{\Z^*}$ so that $z_0^*(z_0)=1$, let $x_0=(e_1, z_0) \in S_{c_0\oplus_\infty \Z}$ where $e_i$ is the canonical basis of $c_0$ and $x_0^*\in S_{(c_0\oplus_\infty \Z)^*}$ be a functional such that $
x_0^*(y,z)= z_0^*(z)$ for arbitrary $(y,z)\in c_0 \oplus_\infty \Z$.

It is clear that $(x_0, x_0^*)\in \Pi(X)$ and, by hypothesis, given $0<\e<\frac{1}{2}$ we have $\eta(\e,(x_0, x_0^*))$ which is written in Remark~\ref{1} as $\tilde{\eta}(\e,(x_0, x_0^*))$. Then, we can take $N$ such that 
$$
z_0^*((z_0(1), \dots, z_0(N), 0, \dots))>1-\eta(\e, (x_0,x_0^*)),
$$
and (once we fix $N$) we can also choose a sequence $N<m_1<m_2<\cdots$ such that
\begin{equation}\label{subseq}
\frac{1}{2^{m_1+2}}\leq \frac{|z_0(N+1)|^2}{2^{N+1}}, \quad \frac{1}{2^{m_2+2}}\leq \frac{|z_0(N+2)|^2}{2^{N+2}}, \dots.
\end{equation}
Let $T\colon X\to X$ be the operator defined by for $y\in c_0$ and $z\in \Z$ 
\begin{equation*} 
T(y,z)= \left(0,\left(y(1)z_0(1), y(1)z_0(2), \dots, y(1)z_0(N), 0, \dots, 0, \underbrace{\frac{y(m_1)}{2}}_{\text{$m_1$-th coord.}}, 0, \dots, 0,  \underbrace{\frac{y(m_2)}{2}}_{\text{$m_2$-th coord.}}, 0, \dots\right)\right)
\end{equation*} 
and let us show that $\|T\|\leq 1$. Note that for each $y\in B_{c_0}$ and $z\in B_{\Z}$ 
\begin{eqnarray*}
\|T(y,z)\|_{X}&=&\left\|  \left(y(1)z_0(1), y(1)z_0(2), \dots, y(1)z_0(N), 0, \dots, 0, \frac{y(m_1)}{2}, 0, \dots, 0,  \frac{y(m_2)}{2}, 0, \dots\right)  \right\|_{\Z}\\
&\leq& \left\|\left(z_0(1), \dots, z_0(N), 0, \dots, 0, \frac{y(m_1)}{2}, 0, \dots\right)\right\|_\infty + \left( \sum_{i=1}^N \frac{|z_0(i)|^2}{2^i} + \sum_{j=1}^\infty \frac{|y(m_j)|^2}{2^{m_j+2}}\right)^{1/2},
\end{eqnarray*}
where the inequality is due to the fact that $|y(1)|\leq 1$. Since $z_0(1)\geq \frac{1}{2}$ and $|y(m_j)|\leq 1$, it is clear that
$$
\|T(y,z)\|_{X} \leq \|z_0\|_\infty + \left(  \sum_{i=1}^N \frac{|z_0(i)|^2}{2^i} + \sum_{j=1}^\infty \frac{1}{2^{m_j+2}}\right)^{1/2}
$$
and, by \eqref{subseq}, we deduce that
$$
\|T(y,z)\|_{X} \leq \|z_0\|_\infty + \left(\sum_{i=1}^\infty \frac{|z_0(i)|^2}{2^i} \right)^{1/2} = \|z_0\|_{\Z}=1.
$$
Since $x_0^*(T(x_0))>1-\eta(\e, (x_0,x_0^*))$, by hypothesis there is an operator $S\colon X\to X$, $v(S)=1$, such that
$$
x_0^*(S(x_0))=1\quad \text{and}\quad \|T-S\|<\e.
$$
Let $P$ and $Q$ be the projections from $X$ onto $\Z$ and $c_0$, respectively. It is clear that  $PT=T$ and $\|PT-PS\|<\varepsilon$.  Also, by \cite[Lemma~1.2]{P} we have that $v(S)=\max\{v(PS), v(QS)\}=1$ and, since $v(QS)\leq \|QS\|<\varepsilon$, we deduce that $PS$ attains its numerical radius. Hence, we may consider $U=PS$. Following the ideas in \cite{P} (see equation (4) in there) we have 
$$
U(y,z)=Az+By 
$$
with $A\in \mathcal{L}(\Z)$ and $B\in \mathcal{L}(c_0,\Z)$ and $\|A\|<\|B\|$ (this last inequality is due to the fact that $S$ is close to $T$). Then, by \cite[Proposition~2.4]{P} we have $\lim_{n\to \infty}e_n^*(B(e_n))=0$ for the canonical basis $(e_i^*)$ of $\ell_1$. Naming $\tilde{T}\colon c_0\to \Z$ to the operator 
$$
\tilde{T}(y)=\left(y(1)z_0(1), y(1)z_0(2), \dots, y(1)z_0(N), 0, \dots, 0, \frac{y(m_1)}{2}, 0, \dots, 0,  \frac{y(m_2)}{2}, 0, \dots\right),
$$
we have $\|\tilde{T}-B\|<\e$ and $e_{m_j}^*(\tilde{T}(e_{m_j}))=\frac{1}{2}$ for all $j\in \mathbb{N}$. Consequently,
$$
\frac{1}{2} = \lim_{j\to \infty} |e_{m_j}^*((\tilde{T}-B)(e_{m_j}))|\leq \e,
$$
which is the desired contradiction.
\end{proof}

We finish the study of property {\bf L}$_{p,p}$-nu by pointing out the immediate fact that a reflexive Banach space $X$ satisfies property {\bf L}$_{p,p}$-nu if and only if its dual $X^*$ satisfies it. This is deduced from the fact that $x^*(Tx)=\hat{x}(T^*x^*)$ and $(T^*)^*=T$ and  $(x,x^*)\in \Pi(X,X^*)$ if $(x, x^*) \in \Pi(X)$, where $\hat{~}:~X \rightarrow X^{**}$ is the canonical isometric inclusion (see \cite[Proposition 3]{KLM}). Notice that this is no longer true if we remove the hypothesis of $X$ being reflexive. Indeed, $c_0$ satisfies the {\bf L}$_{p,p}$-nu by Theorem \ref{c_0} but $\ell_1$ does not by Remark \ref{l_1}.

\begin{prop}A reflexive Banach space $X$ has {\bf L}$_{p,p}$-nu if and only if $X^*$ has {\bf L}$_{p,p}$-nu.
\end{prop}

We now study property {\bf L}$_{o,o}$-nu which, as we will see, turns out to be much more restrictive than property {\bf L}$_{p,p}$-nu. We are starting with finite dimensional Banach spaces. Let us notice that if $\dim(X) < \infty$ and $X$ fails to have {\bf L}$_{o,o}$-nu, we get a contradiction by using the compactness of the unit balls of $X$ and $X^*$. Thus, we have the following result.

\begin{prop}\label{fdim} 
Let $X$ be a finite dimensional Banach space. Then, $X$ has {\bf L}$_{o,o}$-nu.
\end{prop}

In view of the previous result, it is natural to ask whether there are infinite dimensional Banach spaces satisfying {\bf L}$_{o,o}$-nu. Under some general assumptions on the space,  we prove that this is not the case.

\begin{prop}\label{Loo finite dim}
Let $X$ be a Banach space with the approximation property and $n(X) >0$. If $X$ has the {\bf L}$_{o,o}$-nu, then $X$ is finite dimensional.
\end{prop}

 \begin{proof}
Note that, since $X$ has the  {\bf L}$_{o,o}$-nu, then every $T\in \mathcal{L}(X)$ attains its numerical radius and, hence, $X$ is reflexive by \cite[Theorem~1]{AcoRui}.
		Now, in the space $Bil(X\times X^*)$ of bilinear forms from $X\times X^*$ to $\mathbb{K}$, we consider the norm
		$$
		\|\varphi\|_{\Pi}=\sup_{(x,x^*)\in \Pi(X)} |\varphi(x,x^*)|.
		$$

We note that this value becomes a norm from the assumption $n(X)>0$.	Since $X$ is reflexive, the mapping
		\begin{eqnarray}\label{isomorphism}
		\nonumber \left( \mathcal{L}(X), v(\cdot)\right) &\to& \left(Bil(X\times X^*), \|\cdot\|_{\Pi}\right)\\
		T&\mapsto& \varphi_T, \quad \text{with $\varphi_T(x,x^*)=x^*(Tx)$}
		\end{eqnarray}
		is an isometric isomorphism. 

We now claim that $\left(Bil(X\times X^*), \|\cdot\|_{\Pi}\right)$ is a reflexive space. In that case, $\left( \mathcal{L}(X), v(\cdot)\right)$ is reflexive and, since $v(\cdot)$ and $\|\cdot\|$ are equivalent norms in $\mathcal{L}(X)$, $\left( \mathcal{L}(X), \|\cdot\|\right)$ is reflexive. Then, by \cite{Hol} we  have $\mathcal{L}(X)=\mathcal{K}(X)$ and, consequently, $X$ is a finite dimensional space. 
		
		To prove that $\left(Bil(X\times X^*), \|\cdot\|_{\Pi}\right)$ is a reflexive space, consider the (algebraic) subspace of $X\otimes X^*$ given by
		$$
		Z=\left\{ \sum_{i=1}^n \lambda_i x_i\otimes x_i^*:\,\, n\in \mathbb{N}, \lambda_i\in \mathbb{K}, (x_i,x_i^*)\in \Pi(X)\right\}
		$$
		endowed with the norm 
		$$
		\pi_{\Pi}(u)=\inf\left\{\sum_{i=1}^n |\lambda_i| \|x_i\| \|x_i^*\| \right\}
		$$
		where the infimum is taken over all the representations of $u$ of the form $\sum_{i=1}^n \lambda_i x_i\otimes x_i^*$ with $(x_i,x_i^*)\in \Pi(X)$. Let $Y$ be the completion of $Z$ (with the norm $\pi_{\Pi}(\cdot)$) and 
		let us note that $\left(Bil(X\times X^*), \|\cdot\|_{\Pi}\right)$ and $Y^*=(Y, \pi_{\Pi}(\cdot))^*$ are isometrically isomorphic. Consider the (linear) mapping
		\begin{eqnarray}\label{isometry}
		\nonumber \left(Bil(X\times X^*), \|\cdot\|_{\Pi}\right) &\to& Y^*\\
		\varphi&\mapsto& {L_{\varphi}}_{{|_Y}},
		\end{eqnarray}
		where $ {L_{\varphi}}_{{|_Y}}$ is the restriction to $Y$ of the functional $L_{\varphi}\in (X\hat{\otimes}_\pi X^*)^*$ associated to $\varphi$. Noting that 
		$$
		\|{L_{\varphi}}_{{|_Y}}\|_{Y^*} = \sup_{\pi_\Pi(u)=1} |L_{\varphi}(u)|,
		$$
		and it is easy to check that $\|{L_{\varphi}}_{{|_Y}}\|_{Y^*}=\|\varphi\|_{\Pi}$. Hence, the mapping in \eqref{isometry} is an isometry. It remains to prove that is surjective. Given $L\in Y^*$ it is clear that $L_{|_Z}\in Z^\#$ (the algebraic dual of $Z$), and we can consider $\tilde{L}$ an algebraic extension of $L$ to the vector space $X\otimes X^*$, and $\varphi_{\tilde{L}}$ the (non-necessarily bounded) bilinear form associated to $\tilde{L}$. Now, since 
		$$
		\|\varphi_{\tilde{L}}\|_{\Pi} =\sup_{(x,x^*)\in \Pi(X)} |\tilde{L}(x\otimes x^*)|=\sup_{(x,x^*)\in \Pi(X)} |L(x\otimes x^*)| \leq \|L\|_{Y^*},
		$$ 
		we see that  $\varphi_{\tilde{L}}\in  \left(Bil(X\times X^*), \|\cdot\|_{\Pi}\right)$ and that ${L_{\varphi_{\tilde{L}}}}_{|_Z}=\tilde{L}_{|_Z}=L_{|_Z}$. Then, by a continuity argument (note that $\tilde{L}_{|_Z}$ and $L_{|_Z}$ are bounded) we deduce that ${L_{\varphi_{\tilde{L}}}}_{|_Y}=L_{|_Y}$, which proves that the mapping in \eqref{isometry} is surjective.
		
		Since every $T\in \mathcal{L}(X)$ attains its numerical radius, every $\psi \in Bil(X\times X^*)$ attains the $\|\cdot\|_{\Pi}$-norm and, consequently, every functional in $Y^*$ is norm-attaining. Then, by James' theorem, $Y$ is reflexive and, hence,  $\left(Bil(X\times X^*), \|\cdot\|_{\Pi}\right)$ is reflexive.
 \end{proof}
We do not know what happens in the general case. We finish the paper by highlighting this open question.

\begin{question} Let $X$ be any Banach space. If $X$ has the {\bf L}$_{o,o}$-nu, then $X$ is finite dimensional?
\end{question}

{\bf Acknowledgments:} The authors are grateful to Mingu Jung and Miguel Mart\'in for fruitful conversations on the topic of the paper.

\end{document}